\documentclass[options]{amsart}

\usepackage{latexsym,ifthen,amssymb}
\DeclareMathOperator{\dom}{dom}

\title{$\mathsf{RT}_2^2$ does not imply $\mathsf{WKL}_0$}

\keywords{recursion theory, computability theory, reverse mathematics, Ramsey's theorem, weak Konig lemma,
Mathias forcing}
\subjclass{03F35, 03C62, 03D30, 03D80}

\author{Jiayi Liu }

\address{Department of Mathematics, Central South University, South campus, south dormitory No. 6 Room. 619,
Changsha, 410083, China}

\email{g.jiayi.liu@gmail.com}

\thanks{
The present version this paper depends heavily on a write-up composed by Denis Hirschfeldt and Damir Dzhafarov. Therefore I am greatly indebted to them for lending their draft to me freely. I'm also very grateful to them for reviewing my paper with great effort and patience. I'd also like to thank to Peter Cholak for his review of my paper.
}
\thanks{I am especially grateful to Liang Yu for his invitation for me to attend The 2011 International Workshop of Logic meeting held in ZheJiang Normal University which made this paper known to others. I'm indebted to Chi Tat Chong and Yue Yang for their help in applying for the Computational Prospects of Infinity II: AII Graduate Summer School (2011), although due to personal reasons I'm not able to fulfill the journey.
I'm also grateful to Antonio Montalb$\acute{a}$n for his invitation for me to attend the 2011's Reverse Mathematics Workshop meeting held in University of Chicago.}

\thanks{I'm supported by ZheJiang Normal University Department of Mathematics and NanJing University Department of Mathematics during my journey to ZheJiang Normal University. I wish to thank these institutions. }

\thanks{The author's temporal but legal name is Lu Liu.}

\newcommand{\uhr}{\upharpoonright}

\newcommand{\sub}[1]{_{\textrm{\tiny{\fontfamily{cmr}\selectfont #1}}}}

\newtheorem{theorem}{Theorem}[section]

\newtheorem{lemma}[theorem]{Lemma}
\newtheorem*{Ramsey's theorem}{Ramsey's theorem}
\newtheorem*{claim}{Claim}
\newtheorem{corollary}[theorem]{Corollary}

\theoremstyle{definition}
\newtheorem{definition}[theorem]{Definition}
\newtheorem{fact}[theorem]{Fact}

\begin{document}

\begin{abstract}
We prove that $\mathsf{RCA}_0+\mathsf{RT}_2^2\not\rightarrow \mathsf{WKL}_0$ by showing that for any set $C$ not of \textrm{PA}-degree and any set $A$, there exists an infinite subset $G$ of $A$ or $\overline{A}$, such that $G\oplus C$ is also not of \textrm{PA}-degree.

\end{abstract}

\maketitle

\section{Introduction}
\label{sec1}

Reverse mathematics studies the proof theoretic strength of various second order arithmetic
statements. Several statements are so important and fundamental that they serve as level lines. Many mathematical theorems are found to be equivalent to these statements and they are unchanged under small perturbations of themselves. The relationships between
these statements and "other" statements draw large attention. $\mathsf{WKL}_0$ is one of these statements. $\mathsf{WKL}_0$ states that every infinite binary
tree admits an infinite path. It is well known that as a second order arithmetic statement, $\mathsf{WKL}_0$ is equivalent to the
statement that for any set $C$ there exists $B\gg C$, where $B\gg C$ means $B$ is of \textrm{PA}-degree relative to $C$.
A good survey of reverse mathematics is \cite{simpson1999subsystems} or \cite{friedman1974some}, \cite{friedman1976systems}. One of the second order
arithmetic statements close to $\mathsf{WKL}_0$ is $\mathsf{RT}_2^2$.

\begin{definition}
Let $[X]^k$ denote $\{F\subseteq X: |F|=k\}$. A k-coloring $f$ is a function, $[X]^n\rightarrow \{1,2,\ldots, k\}$. A
set $H\subseteq [X]^k$ is homogeneous for $f$ iff $f$ is constant on $[H]$. A stable coloring $f$ is a
%"coloring" to "2-coloring","\mathbb{N}" to "[\mathbb{N}]^2"%
2-coloring of $[\mathbb{N}]^2$ such that $(\forall n\in\mathbb{N})(\exists N)(\forall m>N)$
$f(\{m,n\})=f(\{N,n\})$. For a stable coloring $f$, $f_1=\{n\in\mathbb{N}: (\exists N)(\forall
m>N),f(m,n)=1\}$, $f_2=\mathbb{N}-f_1$.
\end{definition}

\begin{Ramsey's theorem}[Ramsey \cite{ramsey1930problem}]
 For any n and k, every k-coloring of $[\mathbb{N}]^n$ admits an infinite homogeneous set.
 \end{Ramsey's theorem}

Let $\mathsf{RT}_k^n$ denote the Ramsey's theorem for k-coloring of $[\mathbb{N}]^n$. And $\mathsf{SRT}_k^2$ denotes the Ramsey's theorem
restricted to stable coloring of pair.

 Jockusch \cite{jockusch1972ramsey} showed that for $n>2$ $\mathsf{RT}_2^n$ is equivalent to $\mathsf{ACA}_0$, while Seetapun and Slaman
 \cite{seetapun1995strength} showed that $\mathsf{RT}_2^2$ does not imply $\mathsf{ACA}_0$. As to $\mathsf{WKL}_0$, Jockusch
 \cite{jockusch1972ramsey} proved that $\mathsf{WKL}_0$ does not imply $\mathsf{RT}_2^2$. Whether $\mathsf{RT}_2^2$ implies $\mathsf{WKL}_0$ remained open. A more detailed survey of Ramsey's theorem in view of reverse mathematics can be found in Cholak Jockusch and Slaman
 \cite{Cholak2001}. Say a set $S$ cone avoid a class $\mathcal{M}$ iff $(\forall C\in\mathcal{M})[ C\not\leq_T S]$.

The problem has been a major focus in reverse mathematics in the past twenty years. The first important progress was made by Seetapun and Slaman
\cite{seetapun1995strength}, where they showed that
\begin{theorem}[Seetapun and Slaman \cite{seetapun1995strength}]
For any countable class of sets $\{C_j\}$ $j\in\omega$, each $C_i$ is non-computable, then any computable
2-coloring of pairs admits an infinite cone avoiding (for $\{C_j\}$) homogeneous set.
\end{theorem}
\noindent Parallel this result, using Mathias Forcing in a different manner Dzhafarov and Jockusch \cite{Dzhafarov2009} Lemma 3.2 proved that
\begin{theorem}[Dzhafarov and Jockusch \cite{Dzhafarov2009}]
For any set $A$ and any countable class $\mathcal{M}$, each member of $\mathcal{M}$ is non-computable, there exists an
infinite set $G$ contained in either $A$ or its complement such that $G$ is cone avoiding for $\mathcal{M}$.
\end{theorem}
\noindent The main idea is to restrict the computational complexity (computability power) of the homogeneous set as much as possible, with complexity measured by various measurements. Along this line, with simplicity measured by extent of
lowness, Cholak Jockusch and Slaman \cite{Cholak2001} Theorem 3.1 showed, by a fairly ingenious argument,
\begin{theorem}[Cholak, Jockusch, and Slaman \cite{Cholak2001}]
\label{th1}
For any computable coloring of the unordered pairs of natural numbers with finitely many colors, there is an infinite
$low_2$ homogeneous set $X$.
\end{theorem}
\noindent Here we adopt the same idea to prove that
\begin{theorem}
\label{th2}
For any set $C$ not of \textrm{PA}-degree and any set $A$. There exists an infinite subset $G$ of $A$ or $\overline{A}$, such that $G\oplus C$ is also not of \textrm{PA}-degree.
\end{theorem}

\begin{corollary}
$\mathsf{RT}_2^2\not\rightarrow \mathsf{WKL}_0$
\end{corollary}
\begin{proof}
It suffices to construct a countable class $\mathcal{M}$ satisfying the following four conditions (a)$C,B\in\mathcal{M}\rightarrow
C\oplus B\in\mathcal{M}$; (b)$(C\in\mathcal{M}\wedge B\leq_T C)\rightarrow B\in\mathcal{M}$; (c)$(\forall
C\in\mathcal{M})[C\not\gg 0]$; (d)$\mathcal{M}\vdash \mathsf{RT}_2^2$.
It is shown in \cite{Cholak2001} Lemma 7.11 that $\mathsf{RCA}_0+\mathsf{RT}_2^2$ is equivalent to $\mathsf{RCA}_0+\mathsf{SRT}_2^2+\mathsf{COH}$. Moreover, it's
easy to prove that for any $C$-uniform sequence $C_1,C_2,\ldots$, $C$ being
non-\textrm{PA}-degree, there exists an infinite set $G$ cohesive for $C_1,C_2,\ldots$ such that $G\oplus C$ is not of \textrm{PA}-degree. This can be proved using finite extension method as following. Here and below $\sigma\prec\rho$ means $\sigma$ is an initial part of $\rho$; $\sigma\subseteq \rho$ means $\{n\leq |\sigma|: \sigma(n) = 1\}\subseteq \{n\leq|\rho|:\rho(n)=1\}$.
At stage $s$, we define $Z_s= \begin{cases} Z_{s-1}\cap C_s \textrm{ if }Z_{s-1}\cap C_s\textrm{ is infinite;}\\ Z_{s-1}\cap \overline{C_s}\textrm{ else; }\end{cases}$ ($Z_0=C_0$ if $C_0$ is infinite $\overline{C_0} $ if else), $\rho_s\succ\rho_{s-1}$ with $\rho_{s-1}\subsetneq\rho_s\subseteq Z_s/\rho_{s-1}$. And whenever possible we also require $(\exists n)[\Phi_s^{C\oplus \rho_s}(n)=\Phi_n(n)\downarrow]$. We argue $G=\cup\rho_s$ is one of the desired sets. Clearly $G$ is infinite since $(\forall s)\rho_{s-1}\subsetneq\rho_{s}$. The cohesiveness of $G$ follows from $(\forall s)[G\subseteq^* Z_s]$ and $Z_s\subseteq^* C_s\vee Z_s\subseteq^* \overline{C_s}$. Furthermore, $(\forall s)[\Phi_s^{C\oplus G}\textrm{ is not a 2-DNR}]$. For else, suppose contradictory $\Phi_s^{C\oplus G}$ is a 2-DNR. Therefore $(\forall \rho\succeq\rho_{s-1},\rho\subseteq Z_s/\rho_{s-1})[\Phi_s^{C\oplus \rho}(n)\downarrow\wedge\Phi_n(n)\downarrow \Rightarrow \Phi_s^{C\oplus \rho}(n)\ne\Phi_n(n)]$. Since $\Phi_s^{C\oplus G}$ is total so $(\forall n)(\exists \rho\succeq \rho_{s-1},\rho\subset Z_s/\rho_{s-1})[\Phi_{s}^{C\oplus \rho}(n)\downarrow]$. Thus we could compute a 2-DNR using $Z_s$, but $Z_s\leq_T C$ contradict the fact that $C\not\gg 0$.

Let $B_0=\emptyset$. Let $f\in\Delta_2^{0,B_0}$ be a stable coloring, by Theorem \ref{th2} there exists an infinite
$G_0$, $G_0\subseteq f_1\vee G_0\subseteq f_2$ such that $B_0\oplus G_0\not\gg 0$, note
that such $G_0$ computes an infinite homogeneous set of $f$. Let $B_1=B_0\oplus G_0$, $\mathcal{M}_1 = \{X\in 2^\omega:X\leq_T B_1\}$. Clearly $\mathcal{M}_1$ satisfies (a)(b)(c). Let $G_1$ be cohesive for a sequence of
uniformly $\mathcal{M}_1$-computable sets (where $\mathcal{M}_1$-computable means computable in some
$C\in\mathcal{M}_1$), furthermore $G_1\oplus B_1\not\gg 0$. Let $B_2=B_1\oplus G_1$, $\mathcal{M}_2=\{X\in
2^\omega: X\leq_T B_2\}$. Clearly $\mathcal{M}_2$ also satisfies (a)(b)(c). Iterate the
above process in some way that ensures (1) for any uniformly $\mathcal{M}_j$-computable sequence $C_{1},C_{2}\ldots$, there
exists $G_{i-1}\in\mathcal{M}_i$ cohesive for $C_{1},C_{2},\ldots$ and (2) for any
$C\in\Delta_2^{0,\mathcal{M}_j}$, there exists an infinite $G_{i-1}\in\mathcal{M}_i$, $G_{i-1}\subseteq C\vee
G_{i-1}\subseteq \overline{C}$, while preserving the fact that for all resulted $B_i = B_{i-1}\oplus G_{i-1}$, $B_i\not\gg 0$. It follows that $\mathcal{M}=\bigcup\limits_{i=0}^{\infty}\mathcal{M}_i$ $\vdash \mathsf{RCA}_0+\mathsf{SRT}_2^2\leftrightarrow \mathsf{RT}_2^2$ but clearly $\mathcal{M}$ satisfies (a)(b)(c). The conclusion so follows.
\end{proof}

The organization of this paper is as following. In Section \ref{sec2} we introduce some notations
 and the requirements we use.
In Section \ref{sec3} we give some intuition about the proof by demonstrating the
construction of the first step. Section \ref{sec4} defines the forcing conditions and shows
how to use these conditions to obtain a desired set $G$. Section \ref{sec5} is devoted to
 the most important construction, i.e. how to construct a successive condition to force the
 requirements.

\section{Preliminaries}
\label{sec2}

We say $X$ codes an ordered $k$-partition of $\omega$ iff $X=X_1\oplus X_2\oplus\cdots \oplus X_k$ and $\bigcup\limits_{i=1}^k X_i=\omega$, (\emph{not necessarily} with
$X_i\cap X_j = \emptyset$). A \emph{$k$-partition class} is a non-empty collection of
sets, where each set codes a $k$-partition of $\omega$. A tree $T\subseteq 2^{<\omega}$ is an ordered $k-$partition tree of $\omega$ iff
every $\sigma\in T$ codes an ordered $k$-partition of $\{0,1,\ldots, |\sigma|\}$. Note that the class of all ordered $k-$partitions of $\omega$ is a $\Pi_1^0$ class.

\begin{definition}
For $n$ many ordered $k-$partitions, $X^{0},\ldots, X^{n-1}$
\[
Cross(X^{0},X^{2},\ldots, X^{n-1};2)= \bigoplus\limits_{j< k, p<q\leq n-1}
Y^{(p,q)}_{j}
\]
where $Y^{(p,q)}_{j}=X^{p}_j\cap X^{q}_j$, i.e. $Y^{(p,q)}_{j}$ is the intersection of those
$X^{p}$ and $X^{q}$'s $j^{th}$ part, with $p\ne q$. For $n$ classes of ordered $k-$partitions $S_0,S_1,\ldots, S_{n-1}$

\[\begin{split}
Cross(S_0,S_1,\ldots, S_{n-1};2)=\{& Y\in 2^\omega:\textrm{ there exists } X^{i}\in S_i \textrm{ for each }i\leq n-1,
\\
&Y=Cross(X^{0},\ldots, X^{n-1};2)\}
\end{split}
\]
Note that if each $S_i$ is a $\Pi_1^0$ class, then let $T_i$ be computable tree with $[T_i] = S_i$, operation $Cross$ can be defined on strings of $\{0,1\}$ in a nature way, therefore there exists a computable tree $T\subseteq 2^{<\omega}$ such that
$T=Cross(T_{0},T_{1},\ldots, T_{n-1};2)$. So $[T]=Cross(S_0,S_1,\ldots, S_{n-1};2)$ i.e. $Cross(S_0,S_1,\ldots, S_{n-1};2)$ is a $\Pi_1^0$ class.
\end{definition}

%$set(\rho)$ denote $\{i\in\omega:\rho(i)=1\}$, $\rho/\sigma$ ($Z/\sigma$) is $\rho$ ($Z$) with initial $|\sigma|$ %values
%replaced by $\sigma$.

\begin{definition}
\label{disag}
\begin{enumerate}
\item A \emph{valuation} is a finite partial function $\omega \rightarrow
2$.

\item A valuation $p$ is \emph{correct} if $p(n)\ne\Phi_{n}(n)\!\downarrow$
for all $n \in \dom p$.

\item Valuations $p,q$ are \emph{incompatible} if there is an $n$ such
that $p(n) \neq q(n)$.
\end{enumerate}

\end{definition}

%In the following text, we fix a sequence of sets $C^{j}\in\mathcal{M}$ co-final with $\mathcal{M}$, (i.e $\forall
%C\in\mathcal{M}\exists C^{j},C\leq_T C^{j}$) furthermore $C^{j}\leq_T C^{j+1}$.

We try to ensure that $G$ satisfies the following requirements.
%Call a set $Y$ co-2-\textrm{DNR} iff $(\forall n)[
%\Phi_n(n)\downarrow=i\in\{0,1\}\Rightarrow Y(n)=i]$.

To ensure that $(G \cap A)$ and $(G \cap \overline{A})$ are
infinite, we will satisfy the requirements
$$
Q_m : |G\cap A|\geq m\wedge |G\cap \overline{A}|\geq m.
$$
To ensure that $(G \cap A)
\oplus C$ does not have \textrm{PA}-degree, we would need to satisfy the
requirements
$$
R^A_e : \Phi_e^{(G \cap A) \oplus C} \textrm{ total} \Rightarrow (\exists n)
[\Phi_e^{(G \cap A) \oplus C}(n) = \Phi_n(n)\!\downarrow].
$$
Intuitively, $R^A_e$ is to ensure $(G\cap A)\oplus C$ does not compute any 2-\textrm{DNR} via $\Phi_e$.
(Without loss of generality we assume all $\Phi_0,\Phi_1,\ldots$ in this paper are
$\{0,1\}$-valued functionals.) Similarly, to ensure $(G\cap\overline{A})\oplus C$ does not compute any 2-\textrm{DNR} via $\Phi_e$, we try to make $G$ satisfy
$$
R^{\overline{A}}_i : \Phi_i^{(G \cap \overline{A}) \oplus C} \textrm{
total} \Rightarrow (\exists n)[ \Phi_i^{(G \cap \overline{A}) \oplus
C}(n) = \Phi_n(n)\!\downarrow].
$$
Thus we will satisfy the requirements
$$
R_{e,i}: R^A_e \, \vee \, R^{\overline{A}}_i.
$$

These requirements suffice to provide a desired $G$. Note that if there is some $e$ that $G$ does not satisfy $R^{A}_e$ then $G$ must satisfy all $R_i^{\overline{A}}$ since $G$ satisfy $R_{e,i}$ for all $i$. This implies $G\cap\overline{A}$ is not of \textrm{PA}-degree. See also \cite{Cholak2001}, \cite{Dzhafarov2009}.

Before we introduce the forcing condition, to get some intuition, we firstly demonstrate the construction of the first
step.

\section{First step}
\label{sec3}
%Suppose we want to construct an infinite set $G$, either $(G\cap A)\oplus C$ or $(G\cap\overline{A})\oplus C$ is not of \textrm{PA}-degree.
%\\
Suppose we wish to satisfy $R_{e,i}$ that is:

either $(\exists n)[(\Phi_e^{(G\cap A)\oplus C}(n)= \Phi_n(n)\!\downarrow)\vee \Phi_e^{(G\cap A)\oplus C}$ is not total],

or $(\exists n)[(\Phi_i^{(G\cap \overline{A})\oplus C}(n)= \Phi_n(n)\!\downarrow)\vee \Phi_i^{(G\cap \overline{A})\oplus C}$ is not total].

\medskip

\textbf{Case i.} \emph{Try} to find a correct $p$ such that

\begin{equation}\begin{split}\label{pro1}
&(\forall X=X_0\oplus X_1, X_0\cup X_1=\omega)( \exists \rho\exists n\in \dom p)
\\
&[\Phi_e^{(\rho\cap X_0)\oplus C} (n)\downarrow =\Phi_n(n)\downarrow \ne p(n)\vee \Phi_i^{(\rho\cap X_1)\oplus C} (n)\downarrow =\Phi_n(n)\downarrow\ne p(n)]
\end{split}
\end{equation}
\\
 Note that substitute $X_0=A,X_1=\overline{A}$ in above sentence, there
is a $\rho\in2^{<\omega}$ such that $\Phi_e^{(\rho\cap A)\oplus C} (n)\downarrow = \Phi_n(n)\downarrow \vee \Phi_i^{(\rho\cap \overline{A})\oplus C} (n)\downarrow = \Phi_n(n)\downarrow$. Therefore finitely extend initial segment requirement to $\rho$ and set $P_1=\{\omega\}$. \emph{To satisfy $R_{e,i}$}, we ensure $G\succ \rho$. Clearly all $G\succ \rho$ satisfy $R_{e,i}$.

\medskip

 \textbf{Case ii.} \emph{Try} to find three pairwise incompatible partial functions $p_i:\omega\rightarrow \{0,1\}$, $i=0,1,2$
that ensure the following $\Pi_1^0$ classes are non-empty:
\[
\begin{split}
S_{i}=&\{X=X_0\oplus X_1: X_0\cup X_1=\omega \wedge
\\
&[(\forall Z)(\forall n\in \dom p_i)\ \neg(\Phi_e^{(Z\cap X_0)\oplus C}(n)\!\downarrow\ne p_i(n))\wedge \neg(\Phi_i^{(Z\cap X_1)\oplus C}(n)\!\downarrow \ne p_i(n))\ ]
\}
\end{split}
\]
Let
\[
P_1=Cross(S_{0},S_1,S_2;2)
\]

\noindent i.e.

$(\forall Y\in P_1)\ Y=Y_0\oplus Y_1\oplus Y_2\oplus Y_3\oplus Y_4\oplus Y_5$

$(\exists X^0\in S_0\ \exists X^1\in S_1\ \exists X^2\in S_2)$ $X^i = X^i_0\oplus X^i_1$ for $i=0,1,2$ such that
\\

$Y_0=X^0_0\cap X^1_0, $ %for some $X^1\in [T_{(H^{1},v^{1})}],X^2\in [T_{(H^{2},v^{2})}]$
$Y_1=X^1_0\cap X^2_0, $ %for some $X^2\in [T_{(H^{2},v^{2})}],X^3\in [T_{(H^{3},v^{3})}]$
$Y_2=X^2_0\cap X^0_0, $ %for some $X^3\in [T_{(H^{3},v^{3})}],X^1\in [T_{(H^{1},v^{1})}]$
\\

$Y_3=X^0_1\cap X^1_1, $ %for some $X^1\in [T_{(H^{1},v^{1})}],X^2\in [T_{(H^{2},v^{2})}]$
$Y_4=X^1_1\cap X^2_1, $ %for some $X^2\in [T_{(H^{2},v^{2})}],X^3\in [T_{(H^{3},v^{3})}]$
$Y_5=X^2_1\cap X^0_1, $\\ %for some $X^3\in [T_{(H^{3},v^{3})}],X^1\in [T_{(H^{1},v^{1})}]$

 Note:
\begin{enumerate}
\item $S_i$ is a $\Pi_1^0$ class of
    ordered 2-partitions for all $i\leq 2$;
\\
\item $\Phi_e^{G\oplus C}$ is not total on any $G\subseteq Y_{i}$, for $i=0,1,2$ and $\Phi_i^{G\oplus C}$ is not total on any $G\subseteq Y_{i}$, for $i=3,4,5$.
To see this, suppose for some $G\subseteq Y_0$, $\Phi_e^{G\oplus C}$ outputs on both $\dom p_0,\dom p_1$. Let $p_0(n)\ne p_1(n)$ then either $\Phi_e^{G\oplus C}(n)\ne p_0(n)$ or $\Phi_e^{G\oplus C}(n)\ne p_1(n)$. (Recall that we assume that all $\Phi$ are $\{0,1\}-$valued.) Suppose it is the former case, but $G\subseteq Y_0\subseteq
X^0_0$, $X^0_0\oplus X^0_1\in S_0$, by definition of $S_0$
$\Phi_e^{G\oplus C}(n)\downarrow\Rightarrow \Phi_e^{G\oplus C}(n)= p_0(n)$;
\\
\item $P_1$ is a $\Pi_1^0$ class. Though seemingly not, but note that each $S_i$ is a $\Pi_1^0$
    class therefore there are computable trees $T_i$, $i\leq 2k$, such that
    $[T_i]=S_i$ for all $i$, furthermore $Cross$ can be applied to binary strings and is
    computable in this sense, thus there exists some computable tree
    $T'_1=Cross(T_0,T_1,T_2;2)$ with $P_1=[T'_1]$;
\\
\item $\bigcup_{i=0}^5 Y_i=\omega$. (See Lemma \ref{lem1}, this is just the pigeonhole principle. This is why we choose
    \emph{three} pairwise incompatible valuations at \emph{this} step.)
\end{enumerate}

\emph{To satisfy $R_{e,i}$}, we ensure that for some path $Y\in P_1$, $Y=Y_0\oplus Y_1\oplus \cdots \oplus Y_5$, $G$ will be contained in some $Y_i$. By item 2 in above note, $R_{e,i}$ is satisfied.

\medskip

We will show in Lemma \ref{altlem} that if there is no correct valuation as in case i then there must exist such three incompatible valuations i.e., either case i or case ii occurs.

Now we give the framework of our construction i.e. the forcing conditions.

\section{Tree forcing}
\label{sec4}
Let $\sigma \in 2^{<\omega}$ and let $X$ be either an element of
$2^\omega$ or an element of $2^{<\omega}$ of length at least the same
as that of $\sigma$. Here and below, we write $X/\sigma$ for the set
obtained by replacing the first $|\sigma|$ many bits of $X$ by
$\sigma$.

We will use conditions that are elaborations on Mathias forcing
conditions. Here a \emph{Mathias condition} is a pair $(\sigma,X)$
with $\sigma \in 2^{<\omega}$ and $X \in 2^\omega$. The Mathias
condition $(\tau,Y)$ \emph{extends} the Mathias condition $(\sigma,X)$
if $\sigma \preceq \tau$ and $Y/\tau \subseteq X/\sigma$.  A set $G$
\emph{satisfies} the Mathias condition $(\sigma,X)$ iff $\sigma \prec
G$ and $G \subseteq X/\sigma$.

We will be
interested in $\Pi^{0,C}_1$ $k$-partition classes, that is,
$\Pi^{0,C}_1$ classes that are also $k$-partition classes.

\begin{definition}

 A \emph{condition} is a tuple of the form
$(k,\sigma_0,\ldots,\sigma_{k-1},P)$, where $k>0$, each $\sigma_i \in
2^{<\omega}$, and in this paper $P$ is a non-empty $\Pi^{0,C}_1$ $k$-partition class. We think of
each $X_0 \oplus \cdots \oplus X_{k-1} \in P$ as representing $k$ many
Mathias conditions $(\sigma_i,X_i)$ for $i < k$.
\end{definition}

\begin{definition}
\label{def-ext}
 A condition
 $$d=(m,\tau_0,\ldots,\tau_{m-1},Q) \textrm{\emph{ extends }}
c=(k,\sigma_0,\ldots,\sigma_{k-1},P),$$
also denoted by $d\leq c$,
iff there is a function $f : m\rightarrow k$ with the following property:
for each $Y_0 \oplus\cdots \oplus Y_{m-1} \in Q$ there is an $X_0 \oplus \cdots \oplus
X_{k-1} \in P$ such that each Mathias condition $(\tau_i,Y_i)$ extends
the Mathias condition $(\sigma_{f(i)},X_{f(i)})$. In this case, we say
that $f$ \emph{witnesses} this extension, and that \emph{part $i$ of
$d$ refines part $f(i)$ of $c$}. (Whenever we say that a condition
extends another, we assume we have fixed a function witnessing this
extension.)
\end{definition}

\begin{definition}
\label{def-sat}
A set $G$ \emph{satisfies} the condition
$(k,\sigma_0,\ldots,\sigma_{k-1},P)$ iff there is an $X_0 \oplus
\cdots \oplus X_{k-1} \in P$ such that $G$ satisfies some Mathias
condition $(\sigma_i,X_i)$. In this case, we also say that $G$
satisfies this condition \emph{on part $i$}.
\end{definition}

\begin{definition}
\label{def-for}
\begin{enumerate}
\item A condition $(k,\sigma_0,\ldots,\sigma_{k-1},P)$ \emph{forces $Q_m$
on part $i$} iff $|\sigma\cap A|\geq m \wedge |\sigma\cap \overline{A}|\geq m$. Clearly, if $G$ satisfies
such a condition on part $i$, then $G$ satisfies requirement $Q_m$. (Note
that if $c$ forces $Q_m$ on part $i$, and part $j$ of $d$ refines
part $i$ of $c$, then $d$ forces $Q_m$ on part $j$.)

\item A condition \emph{forces $R_{e,i}$ on part $j$} iff every $G$
satisfying this condition on part $j$ also satisfies requirement
$R_{e,i}$. A condition \emph{forces $R_{e,i}$} iff it forces $R_{e,i}$ on
each of its parts. (Note that if $c$ forces $R_{e,i}$ on part $i$, and
part $j$ of $d$ refines part $i$ of $c$, then $d$ forces $R_{e,i}$ on
part $j$. Therefore, if $c$ forces $R_{e,i}$ and $d$ extends $c$, then $d$
forces $R_{e,i}$.)
\end{enumerate}
\end{definition}

\begin{definition}
For a condition $c=(k,\sigma_0,\ldots,\sigma_{k-1},P)$, we say
that \emph{part $i$ of $c$ is acceptable} if there is an $X_0
\oplus \cdots \oplus X_{k-1} \in P$ such that $X_i \cap A$ and $X_i
\cap \overline{A}$ are both infinite.
\end{definition}

For example, in the first step, $P_0 = \{\omega\},k_0 = 1,\sigma_0=\lambda$, and for every $Y\in P_1$ $Y$ is of the form $Y=\bigoplus_{i=0}^5 Y_i$. Clearly $Y_i\subseteq X_{f_1(i)} $, where $X_{f_1(i)} = \omega\in P_0$. $f_1(i) = 0$ for all $i$ witnesses this extension relation.

Note that it is \emph{not} the case that for every $X'\in P'$ there exists a single $X\in P$ such that $(\forall i\leq
k'-1)[ (\sigma_{i}',X_i')$ $\leq $ $(\sigma_{f(i)},X_{f(i)})]$.

\subsection{The general plan.}
The proof will consist of establishing the following two lemmas. The
proof of the second lemma is the core of the argument.

\begin{lemma}
\label{P-lem1}
Every condition has an acceptable part. Therefore for every condition $c$ and every $m$, there is a condition $d$
extending $c$ such that $d$ forces $Q_m$ on each of its acceptable
parts.
\end{lemma}

\begin{lemma}
\label{R-lem}
For every condition $c$ and every $e$ and $i$, there is a condition $d$
extending $c$ that forces $R_{e,i}$.
\end{lemma}

\noindent\emph{Proof of Theorem \ref{th2}.}

Given these lemmas, it is easy to see that we can build a sequence of
conditions $c_0,c_1,\ldots$ with the following properties.
\begin{enumerate}

\item Each $c_{s+1}$ extends $c_s$.

\item If $s = \langle e,i \rangle$ then $c_s$ forces $R_{e,i}$.

\item Each $c_s$ has an acceptable part.

\item If part $i$ of $c_s$ is acceptable, then $c_s$ forces
$Q_s$ on part $i$.

\end{enumerate}

Clearly, if part $j$ of $c_{s+1}$ refines part $i$ of $c_s$ and is
acceptable, then part $i$ of $c_s$ is also acceptable. Thus we can
think of the acceptable parts of our conditions as forming a tree
under the refinement relation. This tree is finitely branching and
infinite, so it has an infinite path. In other words, there are
$i_0,i_1,\ldots$ such that for each $s$, part $i_{s+1}$ of $c_{s+1}$
refines part $i_s$ of $c_s$, and part $i_s$ of $c_s$ is
acceptable, which implies that $c_s$ forces $Q_s$ on part $i_s$. Write
$c_s=(k_s,\sigma^s_0,\ldots,\sigma^s_{k_s-1},P_s)$. Let $G=\bigcup_s
\sigma^s_{i_s}$. Let $U_s$ be the class of all $Y$ that satisfy
$(\sigma^s_{i_s},X_{i_s})$ for some $X_0 \oplus \cdots \oplus X_{k_s-1}
\in P_s$. Note that
\begin{itemize}
\item $U_0 \supseteq U_1 \supseteq \cdots$; Since $G\in U_{s+1}\Leftrightarrow (\exists X\in P_{s+1})[G$ satisfies $(\sigma_{i_{s+1}}^s,X_{i_{s+1}})]$ $\Rightarrow (\exists Z\in P_{s})[(\sigma_{i_{s+1}}^{s+1},X_{i_{s+1}})\leq (\sigma_{i_s}^s,Z_{i_s})\wedge$ $G$ satisfies $(\sigma_{i_s}^s,Z_{i_s})]$ $\Leftrightarrow G\in U_s$.
\item Each $U_s$ contains an extension of $\sigma^s_{i_s}$ i.e. $U_s\ne\emptyset$;
\item Each $U_s$ is closed;
\end{itemize}
By compactness of $2^\omega$ $\bigcap\limits_{s=0}^\infty U_s\ne\emptyset$. But clearly
$(\forall Z\in \bigcap\limits_{s=0}^\infty U_s) [Z\succ \sigma^s_{i_s}]$ for all $s$.
Thus $G$ is the unique element of $\bigcap\limits_{s=0}^\infty U_s$. In other words,
$G$ satisfies each $c_s$ on part $i_s$, and hence satisfies all of our requirements.

\section{Proof of Lemma \ref{P-lem1}}

\begin{proof}[Proof of Lemma \ref{P-lem1}.]
It is here that we use the assumption that $A \nleq\sub{T} C$. Let
$c=(k,\sigma_0,\ldots,\sigma_{k-1},P)$ be a condition. Write $P_\tau$ for
the set of all $X \in P$ that extend $\tau$.
\begin{claim}
For each
$\tau=\tau_0\oplus \cdots \oplus \tau_{k-1}$, if $P_\tau \neq \emptyset$
then there is an $X_0 \oplus \cdots \oplus X_{k-1} \in P_\tau$ and an
$i<k$ such that $X_i$ contains elements $m \in A$ and $n \in \overline{A}$
such that $m,n \geq |\tau_i|$.
\end{claim}

Assuming the claim for now, we build a sequence of strings as follows. Let
$\rho^0$ be the empty string. Given $\rho^s=\rho^s_0\oplus \cdots \oplus
\rho^s_{k-1}$ such that $P_{\rho^s}$ is non-empty, let $X=X_0 \oplus
\cdots \oplus X_{k-1} \in P_{\rho^s}$ and $i_s<k$ be such that $X_{i_s}$
contains elements $m \in A$ and $n \in \overline{A}$ with $m,n \geq
|\rho^s_{i_s}|$. Then there is a $\rho_{s+1}=\rho^{s+1}_0\oplus \cdots
\oplus \rho^{s+1}_{k-1} \prec X$ such that, thinking of strings as finite
sets, $\rho^{s+1}_{i_s} \setminus \rho^s_{i_s}$ contains elements of both
$A$ and $\overline{A}$. Now let $Y =\bigcup_s \rho_s$ and let $i$ be such
that $i=i_s$ for infinitely many $s$. Then $Y \in P$ and $Y$ witnesses the
fact that part $i$ of $c$ is acceptable.

Fix $m$. To obtain the desired $d\leq c$ that forces $Q_m$ on each of its acceptable part.
It is enough to show that for the condition
$c=(k,\sigma_0,\ldots,\sigma_{k-1},P)$, if part $i$ of $c$ is
acceptable, then there is a condition $d_0=(k,\tau_0,\ldots,\tau_{k-1},Q)$
extending $c$ such that $d_0$ forces $Q_m$ on part $i$, where the
extension of $c$ by $d_0$ is witnessed by the identity map. (Note that if
part $i$ of $d_0$ is acceptable, then so is part $i$ of $c$.) Then we can
iterate this process, forcing $Q_m$ on each acceptable part in turn,
to obtain the condition $d$ in the statement of the lemma.

So fix an acceptable part $i$ of $c$. Then there is a $\tau \succ
\sigma_i$ with $|\tau\cap A|\geq m$ and $|\tau\cap \overline{A}|\geq m$, and there is an $X_0 \oplus
\cdots \oplus X_{k-1} \in P$ with $\tau \prec X_i/\sigma_i$. Let $Q =
\{X_0 \oplus \cdots \oplus X_{k-1} \in P : \tau \prec X_i/\sigma_i\}$. Let
$d_0 =(k,\sigma_0,\ldots,\sigma_{i-1},\tau,\sigma_{i+1},\ldots,\sigma_{k-1},Q)$.
Then $d_0$ is an extension of $c$, with the identity function $id:k\rightarrow k$
witness this extension and it clearly forces $Q_m$ on part $i$.

Thus we are left with verifying the claim.
\begin{proof}[Proof of the claim]
Assume for a contradiction
that there is a $\tau=\tau_0 \oplus \cdots \oplus \tau_{k-1}$ such
that $P_\tau \neq \emptyset$ and for every $X_0 \oplus \cdots \oplus
X_{k-1} \in P_\tau$ and every $i<k$, either $X_i \uhr_{\geq
|\tau_i|}{} \subseteq A$ or $X_i \uhr_{\geq |\tau_i|}{} \subseteq
\overline{A}$. It is easy to see that $\tau$ has an extension
$\nu=\nu_0 \oplus \cdots \oplus \nu_{k-1}$ such that $P_\nu \neq
\emptyset$ and for each $i<k$, either $\nu_i(m_i)=1$ for some $m_i
\geq |\tau_i|$ or for every $X_0 \oplus \cdots \oplus X_{k-1} \in
P_\nu$, we have $X_i \uhr_{\geq |\tau_i|}{} = \emptyset$. In the
latter case, let $m_i$ be undefined. Let $S_A$ be the set of all $i<k$
such that $m_i$ is defined and is in $A$, and let $S_{\overline{A}}$
be the set of all $i<k$ such that $m_i$ is defined and is in
$\overline{A}$. If $X_0 \oplus \cdots \oplus X_{k-1} \in P_\nu$, then
$X_i \uhr_{\geq |\tau_i|}{} \subseteq A$ for all $i \in S_A$, and $X_i
\uhr_{\geq |\tau_i|}{} \subseteq \overline{A}$ for all $i \in
S_{\overline{A}}$.

We now claim we can compute $A$ from $C$, contrary to hypothesis.  To
see that this is the case, let $T$ be a $C$-computable tree such that
$P_\nu$ is the set of infinite paths of $T$. For $\rho \in T$, write
$T_\rho$ for the tree of all strings in $T$ compatible with $\rho$.
Suppose we are given $n \geq |\tau|$. Let $j>|\nu|$ be such that for
each $\rho = \rho_0\oplus \cdots \oplus \rho_{k-1} $ of length $j$, we
have $n < |\rho_i|$ for all $i<k$. Let $L_A$ be the set of all $\rho
\in T$ of length $j$ such that $\rho_i(n)=1$ for some $i \in S_A$ and
let $L_{\overline{A}}$ be the set of all $\rho \in T$ of length $j$
such that $\rho_i(n)=1$ for some $i \in S_{\overline{A}}$. If $\rho
\in L_A$ and $T_\rho$ has an infinite path then, by the definition of
$S_A$, we have $n \in A$.  Similarly, if $\rho \in L_{\overline{A}}$
and $T_\rho$ has an infinite path then $n \in \overline{A}$. Thus, if
$\rho \in L_A$ and $\rho' \in L_{\overline{A}}$, then at least one of
$T_\rho$ and $T_{\rho'}$ must be finite. So if we $C$-compute $T$ and
start removing form $L_A$ and $L_{\overline{A}}$ every $\rho$ such
that $T_\rho$ is found to be finite, one of $L_A$ or
$L_{\overline{A}}$ will eventually be empty. They cannot both be empty
because $P_\nu$ is non-empty. If $L_A$ becomes empty, then $n \in
\overline{A}$. If $L_{\overline{A}}$ becomes empty, then $n \in A$.

\end{proof}

\end{proof}

We now turn to the proof of Lemma \ref{R-lem}.

\section{Forcing $R_{e,i}$}
\label{sec5}
\begin{definition}
\label{def1}
\begin{enumerate}
\item $\Phi_e^{\rho\oplus C}$ \emph{disagrees} with a valuation $p$ on a set $X$ iff there is a $Y\subseteq X$ and an $n\in\dom p$, $\Phi_e^{Y/\rho\oplus C}(n)\ne p(n)$;

\item Let $c=(k,\sigma_0,\ldots,\sigma_{k-1},P)$ be a condition,
$p$ be a valuation and $U\subseteq \{0,1,\ldots, k-1\}$. We say that
$c$ \emph{disagrees} with $p$ on $U$ if for every $X_0 \oplus \cdots \oplus
X_{k-1} \in P$ and every $Z_0,Z_1,\ldots, Z_{2k-1}$ with $(\forall l)[ X_l=Z_{2l} \cup Z_{2l+1}]$, there is
a $Y$, a $j\in U(c)$, and an $n \in \dom p$ such that either $\Phi_e^{((Y \cap
Z_{2j})/\sigma_j^A) \oplus C}(n)\!\downarrow \neq p(n)$ or $\Phi_i^{((Y
\cap Z_{2j+1})/\sigma_j^{\overline{A}}) \oplus C}(n)\!\downarrow \neq
p(n)$.

\end{enumerate}
\end{definition}

The following facts illustrate the central idea of the construction.

\begin{fact}
\label{fac1}
For two pairwise incompatible valuations $p_0,p_1$, if $\Phi^\rho$ does not disagree with \emph{both} $p_0,p_1$, on set $X$. Then for any $Y\subseteq X$, $\Phi^{Y/\rho}$ is not total on $\dom p_0\cup \dom p_1$.
\end{fact}
%R( replace fact 3.4)%
\begin{fact}
\label{fac2}
If $\Phi^\rho$ does not disagree with $p$ on a set $X$
then for any $Y\subseteq X$, $\Phi^\rho$ does not disagree with $p$ on set $Y$.
\end{fact}
Therefore,
\begin{fact}
\label{fac6}
For two incompatible valuations $p_0$, $p_1$. If $\Phi^\rho$ does not disagree with $p_0$ on a set $X_0$, and does not disagree with
$p_1$ on a set $X_1$ then for any $Y\subseteq X_0\cap X_1$, $\Phi^{Y/\rho}$ is not total on
$\dom p_0\cup\dom p_1$.
\end{fact}

The following lemma tells how to ensure that the tree of each condition is an ordered partition tree.
\begin{lemma}
\label{lem1}
For any $n$ many ordered $2k-$partitions of $\omega$, namely $X^{0}$, $X^{1}$,$\ldots$, $X^{n-1}$, if $n > 2k$ then
$Cross(X^{0},X^1,\ldots, X^{n-1};2)$ is a $2k \binom{n}{2}$-partition. Therefore if $S_0,S_2,\ldots, S_{n-1}$ are $n$
classes of ordered $2k$-partitions of $\omega$ then \\
$Cross(S_0,S_1,\ldots, S_n;2)$ is a class of
$2k\binom{n}{2}$-partition of $\omega$.
\end{lemma}
\begin{proof}
Straightforward by pigeonhole principle. It suffices to show that for any $x\in \omega$, there is some $i\leq
2k-1$, some $X^{p},X^{q},p\ne q$, such that $X^p = \bigoplus_{i=0}^{2k-1} X_i^p$, $X^q = \bigoplus_{i=0}^{2k-1} X_i^q$, $x\in X_i^p\cap X_i^q$. For $i=0,2,\ldots, 2k-1$ let
$F_{i}=\{p\leq n-1:x\in X^p_i\}$. Since each $X^{p}$ is an ordered partition, therefore for each $p$
there exists some $i$ such that $p\in F_{i}$. So $\bigcup\limits_{i=0}^{2k-1} F_{i}=\{0,1,2,\ldots, n-1\}$. But $n>2k$ thus
there is some $i\leq 2k-1$ such that $F_{i}$ contains two elements say $p,q$, thus $x\in X^p_i\cap
X^q_i$.
\end{proof}

\subsection{Construction}
\label{subsec-con}
Fix $e,i$ and a condition $c=(k,\sigma_0,\ldots,\sigma_{k-1},P)$. For any
condition $d$, let $U(d)$ be the set of all $j$ such that part $j$ of $d$
does not force $R_{e,i}$ on part $j$. If $U(d)=\emptyset$ then there is
nothing to prove, so we assume $U(d) \neq \emptyset$. It is clearly enough
to obtain a condition $d$ extending $c$ such that $|U(d)|<|U(c)|$. Then one
could simply iterate this process. Here and below, we
write $\sigma^A$ for the string of the same length as $\sigma$ defined
by $\sigma^A(n)=1$ if{}f $\sigma(n)=1\wedge n \in A$, and
similarly for $\sigma^{\overline{A}}$.

We will use two ways to extend conditions.

\medskip

\emph{Begin construction:}

\medskip

\textbf{Case i.} $c$ disagrees with some correct valuation $p$ on $ U(c)$.

Let $X_0 \oplus \cdots \oplus X_{k-1} \in P$. For $j=0,1,\ldots, k-1$ let $Z_{2j}=X_j \cap
A$ and $Z_{2j+1}=X_j \cap \overline{A}$. By the definition of disagreeing with
a correct valuation on $U(c)$, there exists a $j\in U(c)$, an $n\in \dom p$ and a $Y$
 such that either $\Phi_e^{((Y \cap
Z_{2j})/\sigma_j^A) \oplus C}(n)\!\downarrow = \Phi_n(n)\!\downarrow$ or
$\Phi_i^{((Y \cap Z_{2j+1})/\sigma_j^{\overline{A}}) \oplus C}(n)\!\downarrow
= \Phi_{n}(n)\!\downarrow$. In other words, either $\Phi_e^{(Y/\sigma_j \cap A)
\oplus C}(n)\!\downarrow = \Phi_n(n)\!\downarrow$ or $\Phi_i^{(Y/\sigma_j \cap
\overline{A}) \oplus C}(n)\!\downarrow = \Phi_n(n)\!\downarrow$.

If $\tau$ is a sufficiently long initial segment of $Y$, then for every
$Z$ extending $\tau$, we have either $\Phi_e^{(Z \cap A) \oplus
C}(n)\!\downarrow = \Phi_n(n)\!\downarrow$ or $\Phi_i^{(Z \cap
\overline{A}) \oplus C}(n)\!\downarrow = \Phi_n(n)\!\downarrow$. We may
assume that $\tau \succeq \sigma_j$. Let $Q$ be the class of all $W_0
\oplus \cdots \oplus W_{k-1} \in P$ such that $\tau$, thought of as a
finite set, is a subset of $W_j/\sigma_j$ and let
$d=(k,\sigma_0,\ldots,\sigma_{j-1},\tau,\sigma_{j+1},\ldots,\sigma_{k-1},Q)$. Note that $Q$
is a non-empty $\Pi_1^{0,C}$ class since it contains $X_0 \oplus \cdots \oplus X_{k-1}$.
Clearly $d$ is an extension of $c$, with the identity function $id:k\rightarrow k$
 witnessing this extension relation, and clearly $d$ forces $R_{e,i}$ on part $j$, so that
$|U(d)|<|U(c)|$.

\medskip

\textbf{Case ii.} There are pairwise incompatible valuations $p_0,\ldots,p_{2k}$ such
that $c$ does not disagree with any $p_l$ on $U(c)$. We will show in Lemma \ref{altlem}
that these are the only two cases that will occur.

For each $l<2k$ let $S_l$ be the class of all sets of the form $Z_0 \oplus \cdots \oplus
Z_{2k-1}$ such that $(Z_0 \cup Z_1) \oplus (Z_2 \cup Z_3) \oplus \cdots \oplus (Z_{2k-2}
\cup Z_{2k-1}) \in P$ and for all $j\in U(c)$, every $n\in \dom p_l$, every $Y$
 we have, neither
$\Phi_e^{(Y\cap Z_{2j})/\sigma_j^A \oplus C}(n)\!\downarrow \neq p_l(n)$
nor
$\Phi_i^{(Y\cap Z_{2j+1})/\sigma_j^{\overline{A}} \oplus C}(n)\!\downarrow \neq
p_l(n)$.

Since $c$ does not disagree with any
of the $p_l$ on $U(c)$, all $S_l$ are non-empty.  It is then easy to see that each
$S_l$ is in fact a $\Pi^{0,C}_1$ $2k$-partition class.

Let $Q=Cross( S_0,\ldots,S_{2k};2)$ and
let $$d=\left(2k\binom{2k+1}{2}, \sigma_0,\ldots,\sigma_0,
\sigma_1,\ldots,\sigma_1,
\ldots,\sigma_{k-1},\ldots,\sigma_{k-1},Q\right),$$ where each
$\sigma_i$ appears $2\binom{2k+1}{2}$ many times. We show that $d$ is a condition
extending $c$, and $d$ forces $R_{e,i}$.

    \begin{enumerate}
    \item Since each $S_i$ is non-empty therefore $Q$ is non-empty. Furthermore, since
        each $S_i$ is a $\Pi_1^{0,C}$ class then $Q$ is also a
        $\Pi_1^{0,C}$ class. Because $Cross$, when applied to strings, is computable
        therefore by applying $Cross$ to the $2k+1$ computable trees $T_{i}$ with
        $[T_{i}]=S_i$ one obtains a computable tree $T$ with
        $[T]=Q$.

   \item $Q$ is a class of ordered $2k\binom{2k+1}{2}$-partitions of $\omega$. To see this, note that
       $S_i$, $i\leq 2k$, are $2k+1$ classes of ordered 2k-partitions of $\omega$, by Lemma
       \ref{lem1} $Q$ is a class of ordered $2k\binom{2k+1}{2}$-partitions of $\omega$. Therefore combine with item
       1 and recall the fact that the initial segments in $d$ are not changed, it
       follows that $d$ \emph{is} a condition.

    \item For each new part $i'$ of $d$ and every $ W_0\oplus W_1\oplus\cdots\oplus W_{k'-1}\in Q$, where $k'=2k\binom{2k+1}{2}$, there exists $X_0\oplus X_1\oplus\cdots\oplus X_{k-1}\in P$, and $i\leq k-1$ with $W_{i'}/\sigma_{i'}\subseteq
        X_{i}/\sigma_{i}$, and $\sigma_i=\sigma_{i'}$, i.e. each new part is contained in an old part of some path through $P$. It follows
        that $d$ extends $c$. To see this, note that by definition of $P$ for
        each $i'\leq k'-1$ there exist $p,q\leq 2k,p\ne q$ and $j\leq 2k-1$ determined by $i'$, such that $(\forall W\in Q)( \exists X^{p}\in
        S_p\ \exists X^{q}\in S_q)$
        $[W_{i'}=X_{j}^p\cap X_{j}^q]$. Furthermore, by definition
        of $S_p$,
     $X^p_{j}\cup X^{p}_{j'}= X_i$ for some  $j'\leq 2k-1$, and some $X=X_0\oplus X_1\oplus\cdots\oplus X_{k-1}\in P$.
    Therefore
    $$W_{i'}=X^p_{j}\cap X^q_j\subseteq X^p_j\subseteq X^p_j\cup X^p_{j'}=X_i$$ i.e. each part $i'$ of each $W\in Q$ is contained in some part $i$ of some $X\in P$.

    \item $d$ forces
$R_{e,i}$. To see this, let $G$ satisfy $d$. Then there is some $j<k$, some $a \neq b<2k+1$, some
$Z_0 \oplus \cdots \oplus Z_{2k-1} \in S_a$, and some $W_0 \oplus \cdots
\oplus W_{2k-1} \in S_b$ such that $G$ satisfies one of the Mathias
conditions $(\sigma_j,Z_{2j} \cap W_{2j})$ or $(\sigma_j,Z_{2j+1} \cap
W_{2j+1})$. Then $G$ satisfies $c$ on part $j$, so if $j \notin U(c)$,
then $G$ satisfies $R_{e,i}$. So assume $j \in U(c)$.

Let us suppose $G$ satisfies $(\sigma_j,Z_{2j} \cap W_{2j})$, the other
case being similar. Then $(G \cap A)/\sigma_j$ satisfies both of the
Mathias conditions $(\sigma_j,Z_{2j})$ and $(\sigma_j,W_{2j})$. Let $n$ be
such that $p_a(n) \neq p_b(n)$. By the definitions of $S_a$ and $S_b$, we
have $\neg(\Phi_e^{(G \cap A) \oplus C}(n)\!\downarrow \neq p_a(n))$ and
$\neg(\Phi_e^{(G \cap A)\oplus C}(n)\!\downarrow \neq p_b(n))$. Hence we
must have $\Phi_e^{(G \cap A) \oplus C}(n)\!\uparrow$. Thus $d$ forces
$R_{e,i}$.
\end{enumerate}
\medskip
   \emph{End of construction}
   \medskip

It remains to prove that
   \begin{lemma}
\label{altlem}
For a valuation $p$, let $S_p$ be the $\Pi_1^{0,C}$ class of all $Z_0\oplus\cdots\oplus Z_{2k-1}$ with $Z_{0}\cup Z_{1}\oplus\cdots\oplus Z_{2k-2}\cup Z_{2k-1}\in P$ such that for every $j\in U(c)$, every $\mu\in 2^{\omega}$, and every $n \in \dom p$,
\begin{itemize}
\item neither
$\Phi_e^{((\mu\cap Z_{2j})/\sigma^A_j) \oplus C}(n)[|\mu|]\!\downarrow \neq
p(n)$,
\item nor $\Phi_i^{((\mu\cap Z_{2j+1})/\sigma^{\overline{A}}_j) \oplus
C}(n)[|\mu|]\!\downarrow \neq p(n)$.
\end{itemize}

One of the following must hold.
\begin{enumerate}

\item There is a correct valuation $p$ such that $S_p$ is empty i.e. $c$ disagrees with the correct $p$ on $U(c)$.

\item There are pairwise incompatible valuations $p_0,\ldots,p_{2k}$ such
that $S_p$ is not empty i.e. $c$ does not disagree with $p_l$ on $U(c)$ for all $l\leq 2k$.

\end{enumerate}
\end{lemma}
\begin{proof}[Proof of Lemma \ref{altlem}]
We note that item 1 and item 2 are equivalent to case i and case ii respectively. Furthermore $S_p$ is a  $\Pi_1^{0,C}$ class uniformly in $p$. Consequently for each $j<k$, the set of all valuations $p$ such that $c$
disagrees with $p$ on $U(c)$ is $C$-c.e. Let $E$ denote this $C$-c.e. set of valuations.

Assume that alternative 1 above does not hold. Since $C$ does not have \textrm{PA}-degree,
 there is no $C$-computable function $h$ such that if
$\Phi_n(n)\!\downarrow$ then $h(n)\neq\Phi_n(n)$.

Let $S$ be the collection of all finite sets $F$ such that for each $n
\notin F$, either $\Phi_n(n)\!\downarrow$ or there is a $p \in E$ such
that $F \cup \{n\} \subseteq \dom p$ and for every $m \in \dom p \setminus
F \cup \{n\}$, we have $p(m)\neq \Phi_m(m)\!\downarrow$. If $F \notin S$, then
there is at least one $n \notin F$ for which the above does not hold. We
say that any such $n$ \emph{witnesses} that $F \notin S$.

First suppose that $\emptyset \in S$. Then for each $n$, either
$\Phi_{n}(n)\!\downarrow$ or there is a $p \in E$ such that $n \in
\dom p$ and for every $m \neq n$ in $\dom p$, we have
$p(m)\neq\Phi_m(m)\!\downarrow$. Then we can define $h \leq\sub{T} C$ by
waiting until either $\Phi_n(n)\!\downarrow$, in which case we let
$h(n)=1-\Phi_n(n)$, or a $p$ as above enters $E$, in which case we let
$h(n)=1-p(n)$. Since no element of $E$ is correct, in the latter case,
if $\Phi_n(n)\!\downarrow$ then $p(n) = \Phi_n(n)$, so
$h(n)=\Phi_n(n)$. Since $C$ does not have \textrm{PA}-degree, this case cannot
occur.

Thus $\emptyset \notin S$. Let $n_0$ witness this fact.
Given $n_0,\ldots,n_j$, if $\{n_0,\ldots,n_j\} \notin S$, then let
$n_{j+1}$ witness this fact. Note that if $n_j$ is defined then
$\Phi_{n_j}(n_j)\!\uparrow$.

Suppose that for some $j$, we have $\{n_0,\ldots,n_j\} \in S$. Then
$\{n_0,\ldots,n_{j-1}\} \notin S$, as otherwise $n_j$ would not be
defined. We define $h \leq\sub{T} C$ as follows. First, let $h(n_l)=0$
for $l \leq j$.  Given $n \notin \{n_0,\ldots,n_j\}$, we wait until
either $\Phi_n(n)\!\downarrow$, in which case we let $h(n)=1-\Phi_n(n)$,
or a $p$ enters $E$ such that $\{n_0,\ldots,n_j,n\} \subseteq \dom p$
and for every $m \in \dom p \setminus \{n_0,\ldots,n_j,n\}$, we have
$p(m)\ne\Phi_m(m)\!\downarrow$. If $\Phi_n(n)\!\uparrow$ then the latter
case must occur, since $\{n_0,\ldots,n_j\} \in S$. In this case, we
cannot have $p(n)\neq\Phi_n(n)\!\downarrow$, as then $p$ would be a
counterexample to the fact that $n_j$ witnesses that
$\{n_0,\ldots,n_{j-1}\} \notin S$. Thus we can let
$h(n)=1-p(n)$. Again, since $C$ does not have \textrm{PA}-degree, this case
cannot occur.

Thus $\{n_0,\ldots,n_j\} \notin S$ for all $j$. There are $2^{j+1}$
many valuations with domain $\{n_0,\ldots,n_j\}$, and they are all
pairwise incompatible. None of these valuations can be in $E$, as that
would contradict the fact that $n_j$ witnesses that
$\{n_0,\ldots,n_{j-1}\} \notin S$. Taking $j$ large enough, we have
$2k+1$ many pairwise incompatible valuations, none of which are in $E$.
\end{proof}

\bibliographystyle{plain}

\bibliography{RT2}

\end{document}